\documentclass{amsart}
\usepackage{amssymb}
\usepackage{amsfonts}

\newtheorem{theorem}{Theorem}
\theoremstyle{plain}

\newtheorem{corollary}{Corollary}

\newtheorem{lemma}{Lemma}

\newtheorem{proposition}{Proposition}

\numberwithin{equation}{section}
\input{tcilatex}

\begin{document}
\title[Speedups of ergodic group extensions]{Speedups of ergodic group
extensions}
\author{Andrey Babichev}
\address{Wesleyan University\\
Middletown, CT}
\email{ababichev@wesleyan.edu}
\author{Robert M. Burton}
\address{Oregon State University\\
Corvallis, OR}
\email{dr.bob.math@gmail.com}
\author{Adam Fieldsteel}
\address{Wesleyan University\\
Middletown, CT}
\email{afieldsteel@wesleyan.edu}
\date{December 18, 2011}
\subjclass{}
\keywords{ergodic transformation, cocycle, group extension}
\dedicatory{We dedicate this paper to the memory of our great friend and
teacher, Dan Rudolph.}
\thanks{The authors gratefully acknowledge the assistance of the anonymous
referee, whose contributions enormously improved this work.}

\begin{abstract}
We prove that for all ergodic extensions $S_{1}$ of a transformation by a
locally compact second countable group $G,$ and for all $G-$extensions $%
S_{2} $ of an aperiodic transformation, there is a relative speedup of $%
S_{1} $ that is relatively isomorphic to $S_{2}$. We apply this result to
give necessary and sufficient conditions for two ergodic $n-$point or
countable extensions to be related in this way.
\end{abstract}

\maketitle

\section{Introduction}

Let $\left( X,\mathcal{B},\mu \right) $ be a Lebesgue probability space and $%
T:X\rightarrow X$ an ergodic $\mu -$preserving automorphism. By a \emph{%
speedup} of $T$ we mean an automorphism of $X$ of the form $x\mapsto
T^{p\left( x\right) }\left( x\right) ,$ where $p\ $is a positive
integer-valued function on $X.$ We denote such an automorphism by $T^{p}.$
It is natural to ask which automorphisms (up to isomorphism) can be obtained
from $T$ in this way. If $p$ is integrable there are significant
restrictions on the possible speedups of $T.$ It was proved in \cite{N}, for
example, that if $S$ is isomorphic to $T^{p\left( x\right) }$ and $\int
pd\mu <\infty ,$ then the entropies of $S$ and $T$ satisfy $h\left( S\right)
=\left( \int pd\mu \right) h\left( T\right) .$ As another example, from \cite%
{OW} we see that if $S$ is isomorphic to $T^{p\left( x\right) }$ and $\int
pd\mu <\infty ,$ then $T$ is a factor of a finite measure preserving
transformation that induces $S.$ Thus, if $S$ is loosely Bernoulli, it can
only be expressed as an integrable speedup of $T$ if $T$ is also loosely
Bernoulli. However, if $p$ is not required to be integrable, then there are
no obstructions to this relation; in \cite{AOW} the general result was
proved that for all ergodic finite measure preserving automorphisms $T$ and
all aperiodic finite measure preserving $S,$ there is a speedup of $T$ that
is isomorphic to $S.$ In this paper we prove a conditional version of that
result in the case of ergodic group extensions, and we give an application
of this result to the classification of ergodic finite extensions.

Suppose that $\left( X,\mathcal{B},\mu \right) $ and $T$ are as above, and $%
T $ is a factor of an automorphism $S$ of the space $\left( Y,\mathcal{C}%
,\nu \right) $. Then by a speedup of $S$ \emph{relative to} $T$ we mean a
speedup $S^{p}$ where $p$ is measurable with respect to the factor $\left( X,%
\mathcal{B},\mu \right) $. Of particular interest to us it the case where $S$
is a group extension of $T$ by a locally compact second countable group $G.$
We recall some basic definitions:

Let $G$ be a locally compact second countable group with left Haar measure $%
\lambda $, and let $\sigma :X\times \mathbb{Z}\rightarrow G$ be a cocycle
for $T.$ That is, $\sigma $ is a measurable function such that for almost
all $x\in X$ and all $n,m\in \mathbb{Z}$%
\begin{equation}
\sigma \left( x,m+n\right) =\sigma \left( T^{n}x,m\right) \sigma \left(
x,n\right) .  \label{cocycle condition}
\end{equation}%
From $T$ and $\sigma $ we obtain an automorphism $T_{\sigma }$ of $\left(
X\times G,\mu \times \lambda \right) $ such that for all $n\in \mathbb{Z},$ 
\begin{equation*}
\left( T_{\sigma }\right) ^{n}\left( x,g\right) =\left( T^{n}x,\sigma \left(
x,n\right) g\right)
\end{equation*}%
which has $\left( T,X\right) $ as a factor. We refer to the map $T_{\sigma }$
as a $G-$extension of $T,$ and to $\left( T,X\right) $ as the base factor of
the extension$.$ We write $\sigma ^{\left( n\right) }$ for the function 
\begin{equation*}
\sigma ^{\left( n\right) }:x\mapsto \sigma \left( x,n\right)
\end{equation*}%
and note that $\sigma $ is determined by the function $\sigma ^{\left(
1\right) },$ because of condition (\ref{cocycle condition})$.$

Given a cocycle $\sigma $ for $T$ and a measurable function $\alpha
:X\rightarrow G$ we obtain a new cocycle for $T,$ which we denote by $\sigma
^{\alpha }$, by setting%
\begin{equation*}
\sigma ^{\alpha }\left( x,n\right) =\alpha \left( T^{n}x\right) \sigma
\left( x,n\right) \left( \alpha \left( x\right) \right) ^{-1}.
\end{equation*}%
\bigskip Thus, given two functions $\alpha ,\beta :X\rightarrow G,$ we have 
\begin{equation*}
\left( \sigma ^{\alpha }\right) ^{\beta }=\sigma ^{\beta \alpha }.
\end{equation*}%
Two cocycles $\sigma $ and $\sigma ^{\prime }$ for $T$ are said to be
cohomologous if there exists a measurable function $\alpha :X\rightarrow G$
such that $\sigma ^{\prime }=\sigma ^{\alpha }.$ In this case we say $\sigma
^{\prime }$ is cohomologous to $\sigma $ by the transfer function $\alpha .$

We say that two $G-$extensions $T_{\sigma }$ and $T_{\sigma ^{\prime
}}^{\prime }$ on spaces $X$ and $X^{\prime }$ are $G-$isomorphic if there is
an isomorphism $\Phi :X\times G\rightarrow X^{\prime }\times G$ between $%
T_{\sigma }$ and $T_{\sigma ^{\prime }}^{\prime }$ of the form 
\begin{equation*}
\Phi \left( x,g\right) =\left( \phi \left( x\right) ,\alpha \left( x\right)
g\right)
\end{equation*}%
where $\phi :X\rightarrow X^{\prime }$ is an isomorphism between $T$ and $%
T^{\prime }$ and $\alpha :X\rightarrow G$ is a measurable function$.$ This
is the case precisely when there is an isomorphism $\phi $ between $T$ and $%
T^{\prime }$ such that the cocycle $\sigma ^{\prime }\phi $ for $T$ is
cohomologous to $\sigma $ by the transfer function $\alpha ,$ where $\sigma
^{\prime }\phi $ is given by 
\begin{equation*}
\sigma ^{\prime }\phi \left( x,n\right) =\sigma ^{\prime }\left( \phi
x,n\right) .
\end{equation*}

Given a $G-$extension $T_{\sigma }$ we consider speedups of $T_{\sigma }$
relative to the base factor $\left( T,X\right) $. Each such relative speedup
of $T_{\sigma }$ determines and is determined by a speedup of the factor $T.$
Thus if $\left( T_{\sigma }\right) ^{p}$ is a relative speedup of $T_{\sigma
},$ we have $\left( T_{\sigma }\right) ^{p}=\left( T^{p}\right) _{\sigma
^{\left( p\right) }}$, where $\sigma ^{\left( p\right) }$ is the cocyle for $%
T^{p}$ determined by the values 
\begin{equation*}
\sigma ^{\left( p\right) }\left( x,1\right) =\sigma \left( x,p\left(
x\right) \right) .
\end{equation*}

Our first goal is to prove the following theorem:

\begin{theorem}
\label{gp ext unif}Let $T_{\sigma }$ and $T_{\sigma ^{\prime }}^{\prime }$
be $G-$extensions of $\left( T,X,\mathcal{B},\mu \right) $ and $\left(
T^{\prime },X^{\prime },\mathcal{B}^{\prime },\mu ^{\prime }\right) ,$ where 
$G$ is a locally compact second countable group. Suppose that $T_{\sigma }$
is ergodic and $T^{\prime }$ is aperiodic. Let $U$ be a neighborhood of $%
e_{G}.$ Then there is a relative speedup of $T_{\sigma }$ which is $G-$%
isomorphic to $T_{\sigma ^{\prime }}^{\prime }$ by a $G-$isomorphism whose
transfer function $\alpha $ satisfies $\alpha \left( x\right) \in U$ a.e.
\end{theorem}

We remark that, as shown by Herman and Zimmer $\left( \text{\cite{H}},\text{%
\cite{Z}}\right) ,$ a locally compact second countable group $G$ admits
ergodic $G-$extensions if and only if it is amenable.

Our original proof of this theorem in the case of ergodic $G-$extensions for
compact $G$ \cite{BF} used techniques derived from the restricted orbit
equivalence theory of Rudolph and Kammeyer \cite{R},\cite{KR1},\cite{KR2}$,$
(and so ultimately from Ornstein's isomorphism theorem \cite{O}). However,
as the referee has generously pointed out, a far simpler proof is available
using the methods to be found in \cite{AOW}$,$ which yields a stronger
result, and we present that argument here.

We note that theorem \ref{gp ext unif} may be viewed as an analogue of the
orbit equivalence result for $G-$extensions obtained in \cite{F}, which was
also obtained by other methods in \cite{G}$.$ The point of view which we
take here has much in common with that of Golodets and Sinel'shchikov in
their paper \cite{GS}, which deals with questions of orbit equivalence. In 
\cite{G} a classification of finite extensions up to factor orbit
equivalence was given, and after proving theorem \ref{gp ext unif} we will
adapt the methods of \cite{G} to give an analogous classification of finite
extensions with respect to the speedup relation we have introduced here.

\section{Technical preliminaries}

We use the following terminology. A \emph{Rokhlin tower} $\mathcal{T}$ (or
simply\emph{\ tower}) for an automorphism $T$ on $\left( X,\mathcal{B},\mu
\right) $ is a pairwise disjoint collection $\left\{ A_{i}\right\}
_{i=1}^{h} $ of measurable sets in $X$ such that for each $i,$ $T\left(
A_{i}\right) =A_{i+1}.$ Each $A_{i}\in \mathcal{T}$ is called a \emph{level}
of $\mathcal{T}$, $A_{1}$ is the \emph{base}, $h=h\left( \mathcal{T}\right) $
is the \emph{height}, and the common value $\mu \left( A_{i}\right) $ is the 
\emph{width} $w_{\mathcal{T}}$ of $\mathcal{T}$. We let $\left\vert \mathcal{%
T}\right\vert =\dbigcup_{i=1}^{h}T^{i}A_{1}$ and $\left\vert \mathcal{T}%
\right\vert ^{o}=\dbigcup_{i=1}^{h-1}T^{i}A_{1}$. A \emph{column} of $%
\mathcal{T}$ is a tower of the form $\left\{ T^{i}B\right\} _{i=0}^{h-1}$,
where $B$ is a measurable subset of the base of $\mathcal{T}$.

A \emph{castle} for $T$ is a finite collection $\mathcal{C=}\left\{ \mathcal{%
T}_{j}\right\} _{j=1}^{J}$ of towers for $T$ such that $\left\vert \mathcal{T%
}_{j_{1}}\right\vert \cap \left\vert \mathcal{T}_{j_{2}}\right\vert
=\emptyset $ for all $j_{1}\neq j_{2}.$ We let $\left\vert \mathcal{C}%
\right\vert =\dbigcup_{j=1}^{J}\left\vert \mathcal{T}_{j}\right\vert $ and $%
\left\vert \mathcal{C}\right\vert ^{o}=\dbigcup_{j=1}^{J}\left\vert \mathcal{%
T}_{j}\right\vert ^{o}.$ We refer to $X\backslash \left\vert \mathcal{C}%
\right\vert $ as the \emph{residual} set of $\mathcal{C}$. A\emph{\ level}
of $\mathcal{C}$ (respectively a \emph{column} of $\mathcal{C}$) is a level
(resp. column) of a tower in $\mathcal{C}$. Thus $\dbigcup \left\{ \mathcal{T%
}:\mathcal{T\in C}\right\} $ is the set of all levels of $\mathcal{C}$,
which we denote by $L\left( \mathcal{C}\right) .$

If $\mathcal{T}$ is a tower for $T$ then each finite measurable partition $%
\mathcal{Q}=\left\{ B_{j}\right\} _{j=1}^{J}$ of the base of $\mathcal{T}$
gives rise to a castle $\mathcal{T}_{Q}$ whose towers are the columns of $%
\mathcal{T}$ with bases $B_{j}$. Given a finite partition $\mathcal{P}$ of $%
\left\vert \mathcal{T}\right\vert ,$ we obtain a partition $\mathcal{P}_{%
\mathcal{T}}$ of the base $B$ of $\mathcal{T}$ whose atoms are maximal sets $%
B_{j}$ such that for every $i\in \left\{ 1,2,...,h_{\mathcal{T}}\right\}
,T^{i}B_{j}$ is contained in a single atom of $\mathcal{P}$. That is, $%
\mathcal{P}_{\mathcal{T}}$ is the trace of $\bigvee_{i=0}^{h-1}T^{-i}%
\mathcal{P}$ on $B.$ This partition yields a castle $\left( \mathcal{T}%
\right) _{\mathcal{P}_{\mathcal{T}}}$ as above. We refer to this castle as
the castle of $\mathcal{P-}$columns of $\mathcal{T}$. We make similar
definitions for castles $\mathcal{C}$ and partitions of $\left\vert \mathcal{%
C}\right\vert $ or of the bases of the towers of $\mathcal{C}$. We let $%
\mathcal{P}\left( \mathcal{C}\right) $ denote the partition of $X$ into the
levels of $\mathcal{C}$ and the residual set of $\mathcal{C}.$

Given two castles $\mathcal{C}_{1}$ and $\mathcal{C}_{2}$ for $T,$ we write $%
\mathcal{C}_{1}\leq \mathcal{C}_{2}$ if $\mathcal{C}_{2}$ can be viewed
abstractly as having been obtained from $\mathcal{C}_{1}$ by a cutting and
stacking construction, as in \cite{AOW}$.$ More formally, this means the
following: $\left( i\right) $ $\left\vert \mathcal{C}_{1}\right\vert \subset
\left\vert \mathcal{C}_{2}\right\vert ^{o},$ $\left( ii\right) $ There is a
finite partition $Q$ of the bases of the towers of $\mathcal{C}_{1}$ such
that each level of the castle $\left( \mathcal{C}_{1}\right) _{Q}$ is a
level of $\mathcal{C}_{2},$ and $\left( iii\right) $ for each tower of $%
\left( \mathcal{C}_{1}\right) _{Q}$ there is a tower of $\mathcal{C}_{2}$
that contains it. Note that condition $\left( i\right) $ implies that if $%
\left\{ A_{i}\right\} _{i=1}^{h_{2}}$ is a tower in $\mathcal{C}_{2}$ and $%
A_{j}$ is a base of a tower of $\left( \mathcal{C}_{1}\right) _{Q}$ of
height $h_{1},$ then we must have $j\leq h_{2}-h_{1}.$

We make use of the following lemmas. Lemmas \ref{little push} and \ref{push
forward} are well known, but we include their proofs for the convenience of
the reader.

\begin{lemma}
\label{little push} If $T_{\sigma }$ is an ergodic $G-$extension of $\left(
T,X,\mu \right) $, then given sets $A,B\subset X$ of positive measure, and a
non-empty open set $U\subset G,$ there are a set $A^{\prime }\subset A$ of
positive measure and $n^{\prime }\in \mathbb{N}$ such that $T^{n^{\prime
}}\left( A^{\prime }\right) \subset B$ and for all $x\in A^{\prime },$ $%
\sigma \left( x,n^{\prime }\right) \in U.$
\end{lemma}

\begin{proof}
Fix sets $A,B\subset X$ of positive measure, and a non-empty open set $%
U\subset G.$ Choose non-empty open sets $V_{0}$ and $V_{1}$ in $G$ so that $%
e_{G}\in V_{0}$ and $V_{1}V_{0}^{-1}\subset U.$ Since $T_{\sigma }$ is
ergodic, for almost every $\left( x,g\right) \in A\times V_{0},$ there are
(infinitely many) $n\in \mathbb{N}$ such that 
\begin{equation*}
\left( T_{\sigma }\right) ^{n}\left( x,g\right) \in B\times V_{1}.
\end{equation*}%
Hence there exists some $g_{0}\in V_{0}$ such that for almost all $x\in A,$
there exists $n\in \mathbb{N}$ such that%
\begin{equation*}
\left( T_{\sigma }\right) ^{n}\left( x,g_{0}\right) \in B\times V_{1}.
\end{equation*}%
For each $n\in \mathbb{N},$ let $A_{n}=\left\{ x\in A:\left( T_{\sigma
}\right) ^{n}\left( x,g_{0}\right) \in B\times V_{1}\right\} .$ Then for
some $n^{\prime }\in \mathbb{N}$, $\mu \left( A_{n^{\prime }}\right) >0,$
and so for each $x\in A_{n^{\prime }}$ we have%
\begin{equation*}
\sigma \left( x,n^{\prime }\right) g_{0}\in V_{1}
\end{equation*}%
so%
\begin{equation*}
\sigma \left( x,n^{\prime }\right) \in V_{1}g_{0}^{-1}\subset U.
\end{equation*}%
Setting $A^{\prime }=A_{n^{\prime }}$ we have the desired result.
\end{proof}

For $A^{\prime }$ $B$ and $n^{\prime }$ satisfying the conclusions of this
lemma, we say $\left( A^{\prime },n^{\prime }\right) $ is $\left( B,U\right)
-$good, or simply $A^{\prime }$ is $\left( B,U\right) -$good. We strengthen
this lemma to obtain the following lemma.

\begin{lemma}
\label{push forward}If $T_{\sigma }$ is an ergodic $G-$extension of $\left(
T,X,\mu \right) $, then, given $A,B\subset X$ of equal measure and a
non-empty open set $U\subset G,$ there is a measurable function $%
p:A\rightarrow \mathbb{N}$ such that $T^{p}\upharpoonright A$ is an
isomorphism from $A$ to $B$ and, for almost every $x\in A,$ $\sigma \left(
x,p\left( x\right) \right) \in U.$
\end{lemma}

\begin{proof}
Fix $A,B\subset X$ of equal measure and non-empty open $U\subset G.$ Fix $%
\varepsilon _{i}\downarrow 0$. Let 
\begin{equation*}
a_{1}=\sup \left\{ \mu \left( A^{\prime }\right) :A^{\prime }\subset A\text{
and }A^{\prime }\text{ is }\left( B,U\right) -\text{good}\right\} .
\end{equation*}

Choose $A_{1}\subset A$ and $n_{1}\in \mathbb{N}$ such that $\left(
A_{1},n_{1}\right) $ is $\left( B,U\right) -$good and $\mu \left(
A_{1}\right) >a_{1}-\varepsilon _{1}.$ If $\mu \left( A_{1}\right) =\mu
\left( A\right) ,$ we are done. If $\mu \left( A_{1}\right) <\mu \left(
A\right) $, let 
\begin{equation*}
a_{2}=\sup \left\{ \mu \left( A^{\prime }\right) :A^{\prime }\subset
A\backslash A_{1}\text{ and }A^{\prime }\text{ is }\left( B\backslash \left(
T^{n_{1}}A_{1}\right) ,U\right) -\text{good}\right\}
\end{equation*}%
and choose $A_{2}\subset A\backslash A_{1}$ and $n_{2}\in \mathbb{N}$ such
that $\left( A_{2},n_{2}\right) $ is $\left( B\backslash
T^{n_{1}}A_{1},U\right) -$good and $\mu \left( A_{2}\right)
>a_{2}-\varepsilon _{2}.$

Continue in this way to obtain a pairwise disjoint sequence $\left\{
A_{i}\right\} $ and integers $n_{i}\in \mathbb{N}.$ If, for some $k\in 
\mathbb{N},$ $\mu \left( \bigcup_{i=1}^{k}A_{i}\right) =\mu \left( A\right)
, $ we are done. Suppose, then that for all $k,$ $\mu \left(
\bigcup_{i=1}^{k}A_{i}\right) <\mu \left( A\right) .$ If in fact $\mu \left(
\bigcup_{i=1}^{\infty }A_{i}\right) <\mu \left( A\right) ,$ then by Lemma %
\ref{little push} there is a set $A^{\prime }\subset A\backslash \left(
\bigcup_{i=1}^{\infty }A_{i}\right) $ of positive measure and $n^{\prime
}\in \mathbb{N}$ such that $\left( A^{\prime },n^{\prime }\right) $ is $%
\left( B,U\right) -$good$.$ But $\sum_{i=1}^{\infty }\mu \left( A_{i}\right)
<\infty $, so $\mu \left( A_{i}\right) \rightarrow 0,$ so $\mu \left(
A_{i}\right) +\varepsilon _{i}\rightarrow 0,$ so for some $i$%
\begin{equation*}
a_{i}<\mu \left( A_{i}\right) +\varepsilon _{i}<\mu \left( A^{\prime }\right)
\end{equation*}%
which contradicts the choice of $a_{i}.$ Hence $\mu \left(
\bigcup_{i=1}^{\infty }A_{i}\right) =\mu \left( A\right) ,$ and we are done
in this case as well.
\end{proof}

We note that this lemma can easily be strengthened to say that if a
measurable function $p_{1}:A\rightarrow \mathbb{N}$ is given, then the
function $p$ can be chosen so that in addition to the conclusions of the
lemma, $p\left( x\right) >p_{1}\left( x\right) $ almost everywhere.

In fact, we will use the following stronger form:

\begin{lemma}
\label{push forward function}If $T_{\sigma }$ is an ergodic $G-$extension of 
$\left( T,X,\mu \right) $, then given $A$ and $B\subset X$ of equal positive
measure, $g:A\rightarrow G$ measurable, $p_{1}:A\rightarrow \mathbb{N},$
measurable, and a neighborhood $U$ of $e_{G},$ there is a measurable $%
p:A\rightarrow \mathbb{N},$ with $p\left( x\right) >p_{1}\left( x\right) $
almost everywhere, such that $T^{p}:A\rightarrow B$ is an isomorphism, and $%
\sigma \left( x,p\left( x\right) \right) g\left( x\right) ^{-1}\in U$ almost
everywhere.
\end{lemma}

\begin{proof}
Choose a neighborhood $V$ of $e_{G}$ such that $VV^{-1}\subset U.$ Partition 
$A$ into measurable sets $\left\{ A_{i}\right\} _{i=1}^{\infty }$ such that
for each $i,$ there is some $g_{i}\in G$ such that $g\left( A_{i}\right)
\subset Vg_{i}$. Applying Lemma \ref{push forward}, for each $i,$ choose a
measurable function $q_{i}:A_{i}\rightarrow \mathbb{N}$ with $q_{i}>p_{1}$
on $A_{i}$ so that $T^{q_{i}}:A_{i}\rightarrow B$ is an isomorphism and $%
\sigma \left( x,q_{i}\left( x\right) \right) \in Vg_{i},$ almost everywhere
in $A_{i},$ and so that the sets $\left\{ T^{q_{i}}A_{i}\right\} _{i}$ are
pairwise disjoint. Then for each $x\in A_{i},$ we have $\sigma \left(
x,p\left( x\right) \right) \in Vg_{i}$. But $g\left( x\right) \in Vg_{i},$
so $\sigma \left( x,q_{i}\left( x\right) \right) g\left( x\right) ^{-1}\in
\left( Vg_{i}\right) \left( Vg_{i}\right) ^{-1}=VV^{-1}\subset U.$ Letting $%
p=\bigcup_{i=1}^{\infty }p_{i}$ completes the proof$.$
\end{proof}

\begin{lemma}
\label{castles}Let $G$ be a locally compact second countable group, and let $%
T_{\sigma }$ be a $G-$extension of the aperiodic automorphism $\left( T,X,%
\mathcal{B},\mu \right) $.\ Let $\left\{ U_{k}\right\} _{k=1}^{\infty }$ be
a neighborhood base for $G$ at $e_{G}.$ Then there is a sequence $\left\{ 
\mathcal{C}_{k}\right\} _{k=1}^{\infty }$ of castles, where the towers of $%
\mathcal{C}_{k}$ all have height $h_{k},$ such that:

\begin{enumerate}
\item for each $k,$ $\mathcal{C}_{k}\leq \mathcal{C}_{k+1};$

\item $\mu \left( \dbigcup\limits_{k=1}^{\infty }\left\vert \mathcal{C}%
_{k}\right\vert \right) =1;$

\item $\dbigcup\limits_{k=1}^{\infty }L\left( \mathcal{C}_{k}\right) $
generates $\mathcal{B}$;

\item for each tower $\mathcal{T}$ in $\mathcal{C}_{k}$ with base $A,$ and
each pair of levels $T^{i}A$ and $T^{j}A$ in $\mathcal{T}$, where $1\leq
i<j\leq h_{k},$ there is some $g\in G$ so that for all $x\in A$, $\sigma
\left( x,j-i\right) \in U_{k}g.$
\end{enumerate}
\end{lemma}

\begin{proof}
Fix a sequence of finite partitions $\mathcal{P}_{k}\uparrow \mathcal{B}$ on 
$X$ and a sequence $\varepsilon _{k}\downarrow 0$ with $\sum_{k}\varepsilon
_{k}<1.$ Choose a sequence of towers $\mathcal{T}_{k}$ for $T$\ with
residual sets of measure less than $\varepsilon _{k}$ such that for each $k,$
$\left\vert \mathcal{T}_{k}\right\vert \subset \left\vert \mathcal{T}%
_{k+1}\right\vert ^{o}.$ Denote the base of $\mathcal{T}_{k}$ by $B_{k}$ and
its height by $h_{k}.$ Choose compact $K_{1}\subset G$ so that if 
\begin{equation*}
B_{1}^{\prime }=\left\{ x\in B_{1}\mid \left( \forall i,j\in \left\{
0,...,h_{1}-1\right\} \right) \sigma \left( T^{i}x,j-i\right) \in
K_{1}\right\}
\end{equation*}%
then $\mu \left( B_{1}^{\prime }\right) >\left( 1-\varepsilon _{1}\right)
\mu \left( B_{1}\right) .$ Let $\mathcal{T}_{1}^{\prime }$ be the portion of 
$\mathcal{T}_{1}$ over\thinspace $B_{1}^{\prime }.$ That is, $\mathcal{T}%
_{1}^{\prime }=\left\{ T_{1}^{i}B_{1}^{\prime }\right\} _{i=0}^{h_{1}-1}$.
Partition $K_{1}$ into sets $\left\{ K_{1,i}\right\} _{i=1}^{s_{1}}$ so that
for each $i=1,...,s_{1}$ there exists $g_{1,i}\in G$ with $K_{1,i}\subset
U_{1}g_{1,i}$. Let $\mathcal{K}_{1}:G\rightarrow G$ be given by%
\begin{equation*}
\mathcal{K}_{1}\left( g\right) =\left\{ 
\begin{array}{c}
g_{1,i},\text{ if }g\in K_{1,i} \\ 
e_{G}\text{ if }g\in G\backslash K_{1}%
\end{array}%
\right\} .
\end{equation*}%
Let $\mathcal{Q}_{1}$ be the partition of $B_{1}^{\prime }$ according to the
values of $\left\{ \mathcal{K}_{1}\left( \sigma \left( T^{i}x,j-i\right)
\right) \right\} _{i,j=0}^{h_{1}-1}$ and the values $\left\{ \mathcal{P}%
_{1}\left( T^{i}\left( x\right) \right) \right\} _{i=0}^{h_{1}-1}$. ($%
\mathcal{P}_{1}\left( y\right) $ denotes the partition element containing $y$%
). We denote the resulting castle $\left( \mathcal{T}_{1}^{\prime }\right) _{%
\mathcal{Q}_{1}}$ by $\mathcal{C}_{1}^{\prime }.$

Next consider $\mathcal{T}_{2}$ with base $B_{2}$ and height $h_{2}.\ $%
Choose compact $K_{2}\subset G$ so that if 
\begin{equation*}
B_{2}^{\prime }=\left\{ x\in B_{2}\mid \left( \forall i,j\in \left\{
0,...,h_{2}-1\right\} \right) \sigma \left( T^{i}x,j-i\right) \in
K_{2}\right\}
\end{equation*}%
then $\mu \left( B_{2}^{\prime }\right) >\left( 1-\varepsilon _{2}\right)
\mu \left( B_{2}\right) .$ Let $\mathcal{T}_{2}^{\prime }$ be the the
portion of $\mathcal{T}_{2}$ over $B_{2}^{\prime }.$ Partition $K_{2}$ into
sets $\left\{ K_{2,i}\right\} _{i=0}^{s_{2}}$ so that for each $%
i=1,...,s_{2} $ there exists $g_{2,i}\in G$ with $K_{2,i}\subset U_{2}g_{2,i}
$. Define $\mathcal{K}_{2}:G\rightarrow G$ analogously to $\mathcal{K}_{1}.$
Let $\mathcal{P}_{2}^{\prime }=\mathcal{P}_{2}\vee \mathcal{P}\left( 
\mathcal{C}_{1}^{\prime }\right) ,$ and let $\mathcal{Q}_{2}$ be the
partition of $B_{2}^{\prime }$ according to the values of 
\begin{equation*}
\left\{ \mathcal{K}_{2}\left( \sigma \left( T^{i}x,j-i\right) \right)
\right\} _{i,j=0}^{h_{2}-1}\vee \left\{ \mathcal{P}_{2}^{\prime }\left(
T^{i}\left( x\right) \right) \right\} _{i=0}^{h_{2}-1}.
\end{equation*}
This gives a castle $\mathcal{C}_{2}^{\prime }=\left( \mathcal{T}%
_{2}^{\prime }\right) _{\mathcal{Q}_{2}}.$

Repeating this process produces a sequence of castles $\mathcal{C}%
_{k}^{\prime }$. To obtain condition 1, we restrict each $\mathcal{C}%
_{k}^{\prime }$ to the set $\bigcap_{j=k+1}^{\infty }\left\vert \mathcal{C}%
_{j}^{\prime }\right\vert .$ That is, for each $k$ and each level $A^{\prime
}$ of $\mathcal{C}_{k}^{\prime },$ we replace $A^{\prime }$ by the set $%
A=A^{\prime }\cap \left( \bigcap_{j=k+1}^{\infty }\left\vert \mathcal{C}%
_{j}^{\prime }\right\vert \right) .$ The resulting set of levels is a castle 
$\mathcal{C}_{k}$, and these castles satisfy the conclusions of the
lemma.\bigskip
\end{proof}

\section{The main result}

We now give the proof of Theorem \ref{gp ext unif}.\bigskip

\begin{proof}
Suppose that $T_{\sigma }$ is an ergodic $G-$extension of $\left( T,X,\mu
\right) $ and $T_{\sigma ^{\prime }}^{\prime }$ is a $G-$ extension of the
aperiodic $\left( T^{\prime },X^{\prime },\mu ^{\prime }\right) $. Fix a
neighborhood $U$ of $e_{G},$ which we may assume to be compact. We will
obtain the desired relative speedup of $T_{\sigma }$ and the $G-$isomorphism
from it to $T_{\sigma }^{\prime }$ as limits of a sequence of partially
defined speedups and isomorphisms.

Let $\delta $ be a complete, right-invariant metric on $G$ compatible with
the topology of $G$. (We note that, while such a metric must exist, there
need not be a complete, two-sided invariant metric compatible with the
topology. See \cite{B}.) Fix $\varepsilon >0$ so that $\bar{B}\left(
\varepsilon ,e_{G}\right) ,$ the closed $\delta -$ball of radius $%
\varepsilon $ centered at $e_{G},$ is compact and contained in $U.$ Fix a
sequence $\varepsilon _{k}\downarrow 0$ with $\sum_{k=1}^{\infty
}\varepsilon _{k}<\frac{\varepsilon }{3}.$ For each $k$ choose a compact
neighborhood $U_{k}$ of $e_{G}$ so that $U_{k}U_{k}^{-1}\subset B\left(
\varepsilon _{k},e_{G}\right) .$ Choose a sequence of castles $\left\{ 
\mathcal{C}_{k}^{\prime }\right\} _{k=1}^{\infty }$ for $T_{\sigma ^{\prime
}}^{\prime }$ as in Lemma \ref{castles} with respect to these $U_{k}.$
Denote the towers and levels of these castles by $\mathcal{C}_{k}^{\prime
}=\left\{ \mathcal{T}_{k,j}^{\prime }\right\} _{j}$ and $\mathcal{T}%
_{k,j}^{\prime }=\left\{ A_{k,j,i}^{\prime }\right\} _{i}.$ In particular,
Lemma \ref{castles} gives us, for all $i\in \left\{ 1,2,...,h_{k}-1\right\} $
and for all levels $A_{k,j,i}^{\prime }$ in $\mathcal{C}_{k}^{\prime },$ an
element $g_{k,j,i}\in G$ so that for all $x^{\prime }\in A_{k,j,1}^{\prime }$%
, $\sigma ^{\prime }\left( x^{\prime },i\right) \in U_{k}g_{k,j,i}.$

Make a copy $\mathcal{C}_{1}$ of $\mathcal{C}_{1}^{\prime }$ in $X.$ That
is, choose pairwise-disjoint sets $A_{1,j,i}\in \mathcal{B}$ corresponding
to the levels of $\mathcal{C}_{1}^{\prime }$ such that, for each $j$ and $i,$
$\mu \left( A_{1,j,i}\right) =\mu ^{\prime }\left( A_{1,j,i}^{\prime
}\right) $. Fix $j$ and an isomorphism $\phi _{1,j}:A_{1,j,1}\rightarrow
A_{1,j,1}^{\prime }.$

Applying Lemma \ref{push forward function} repeatedly, we obtain functions $%
q_{i}:A_{1,j,1}\rightarrow \mathbb{N}$ \ with $q_{i}>q_{i-1}$ so that $%
T^{q_{i}}:A_{1}\rightarrow A_{i}$ isomorphically, and for almost every $x\in
A_{1}$ and every $i,$ 
\begin{equation}
\sigma \left( x,q_{i}\left( x\right) \right) \left( \sigma ^{\prime }\left(
\phi _{1,j}\left( x\right) ,i\right) \right) ^{-1}\in B\left( \varepsilon
_{1},e_{G}\right) .  \label{match 1}
\end{equation}%
For each $i\in \left[ 1,h_{1}-1\right] ,$ let $p_{i}:A_{i}\rightarrow 
\mathbb{N}$ be given by setting 
\begin{equation*}
p_{i}\left( x\right) =q_{i+1}\left( T^{-q_{i}}\left( x\right) \right)
-q_{i}\left( T^{-q_{i}}\left( x\right) \right) .
\end{equation*}%
Then, letting $p=\bigcup_{i=1}^{h_{1}-1}p_{i},$ we obtain a partially
defined speedup $T_{1}:=T^{p}$ of $T,$ defined on $\left\vert \mathcal{T}%
_{1,j}\right\vert ^{o},$ for which $\mathcal{T}_{1,j}$ is a tower. This
construction also yields a partially defined cocyle $\sigma _{1}$ for $%
T_{1}, $ which is defined at $\left( x,n\right) $ whenever $x\in \left\vert 
\mathcal{T}_{1,j}\right\vert \cap T_{1}^{-n}\left\vert \mathcal{T}%
_{1,j}\right\vert .$

Extend $\phi _{1,j}$ to $\left\vert \mathcal{T}_{1,j}\right\vert $ so that
on $\left\vert \mathcal{T}_{1,j}\right\vert ^{o}$%
\begin{equation*}
\phi _{1,j}T_{1}\left( x\right) =T^{\prime }\phi _{1,j}\left( x\right) \text{
a.e.}
\end{equation*}%
In particular, for each $i,$ $\phi _{1,j}\left( A_{1,j,i}\right)
=A_{1,j,i}^{\prime }.$ Define $\alpha _{1,j}:\left\vert \mathcal{T}%
_{1,j}\right\vert \rightarrow G$ by setting, for each $x\in A_{1,j,i},$ 
\begin{equation*}
\alpha _{1,j}\left( x\right) =\sigma ^{\prime }\left( \phi _{1,j}\left(
x\right) ,-i\right) ^{-1}\sigma _{1}\left( x,-i\right) .
\end{equation*}

Repeating this construction on each tower of $\mathcal{C}_{1}$, we set $\phi
_{1}=\bigcup_{j}\phi _{1,j}$ to obtain an isomorphism from $\left\vert 
\mathcal{C}_{1}\right\vert $ to $\left\vert \mathcal{C}_{1}^{\prime
}\right\vert $ intertwining $T_{1}$ and $T^{\prime }.$ Similarly, we let $%
\alpha _{1}=\bigcup_{j}\alpha _{1,j}$ and extend $\alpha _{1}$ to $X$ by
setting $\mathcal{\alpha }_{1}\left( x\right) =e_{G}$ for $x\in X\backslash
\left\vert \mathcal{C}_{1}\right\vert .$ We then see that the map $\left(
x,g\right) \mapsto \left( \phi _{1}\left( x\right) ,\alpha _{1}\left(
x\right) g\right) $ is a $G-$isomorphism from $\left( T_{1}\right) _{\sigma
_{1}}$ to $T_{\sigma ^{\prime }}^{\prime },$ insofar as these maps are
defined, which is to say on $\left\vert \mathcal{C}_{1}\right\vert ^{o}.$ In
other words, for all $\left( x,n\right) $ in the domain of $\sigma _{1},$ 
\begin{equation}
\sigma _{1}^{\alpha _{1}}\left( x,n\right) :=\alpha _{1}\left( \left(
T_{1}\right) ^{n}x\right) \sigma _{1}\left( x,n\right) \alpha _{1}\left(
x\right) ^{-1}=\sigma ^{\prime }\left( \phi _{1}\left( x\right) ,n\right) .
\label{cobdy 1}
\end{equation}

We also see that, because of condition $\left( \text{\ref{match 1}}\right) $
and the right-invariance of $\delta ,$ we have for all $x\in X,$ 
\begin{equation}
\delta \left( \alpha _{1}\left( x\right) ,\varepsilon _{G}\right)
<\varepsilon _{1}.
\end{equation}

(We note that in the above construction the approximate constancy of $\sigma
^{\prime }$ on the levels of $\mathcal{C}_{1}^{\prime }$ was not used.)

Now we show how to iterate this construction to complete the proof of the
theorem. Fix an increasing sequence of finite partitions $\left\{ \mathcal{P}%
_{k}\right\} _{k=1}^{\infty }$ of $X$ that generate $\mathcal{B}.$ Choose $%
n_{2}$ so that the partition $\phi _{1}\left( \mathcal{P}_{1}\right) $ is
approximated to within $\frac{1}{2}$ (in the partition metric) by the levels
of $\mathcal{C}_{n_{2}}^{\prime }.$ The index $n_{2}$ must also be chosen so
that $\varepsilon _{n_{2}}$ is small enough to meet an additional condition,
which we will describe at the end of the proof. For notational convenience
re-index $\mathcal{C}_{n_{2}}^{\prime }$ and refer to it as $\mathcal{C}%
_{2}^{\prime },$ and do the same with $\varepsilon _{n_{2}},$ $U_{n_{2}},$
and so on. Let $\mathcal{C}_{2}$ denote a copy of $\mathcal{C}_{2}^{\prime }$
which is the image of $\mathcal{C}_{2}^{\prime }$ under $\phi _{1}^{-1}.$
That is, for each level $A_{2,j,i}^{\prime }$ of $\mathcal{C}_{2}^{\prime }$
contained in $\left\vert \mathcal{C}_{1}^{\prime }\right\vert ,$ the
corresponding level $A_{2,j,i}$ of $\mathcal{C}_{2}$ is given by $%
A_{2,j,i}=\phi _{1}^{-1}\left( A_{2,j,i}^{\prime }\right) .$ Additional
subsets of $X$ are chosen to serve as $A_{2,j,i}$ when $A_{2,j,i}^{\prime }$
is not contained in $\left\vert \mathcal{C}_{1}^{\prime }\right\vert .$

Our goal is to extend $T_{1}$ to a transformation $T_{2}$ on $\left\vert 
\mathcal{C}_{2}\right\vert ^{o}$, so that $T_{2}$ is again a partially
defined speedup of $T,$ with an associated cocycle $\sigma _{2}.$ We will
also modify $\alpha _{1}$ to a function $\alpha _{2}:X\rightarrow G$ so that
on $\left\vert \mathcal{C}_{2}\right\vert ,$ $\alpha _{2}$ serves as a
transfer function for a $G-$isomorphism between $\left( T_{2}\right)
_{\sigma _{2}}$ and $T_{\sigma ^{\prime }}^{\prime }.$

Note that since $U_{2}$ is compact, and there are only finitely many towers
in $\mathcal{C}_{2}^{\prime },$ there is a compact set $K$ so that for all $%
\left( x^{\prime },n\right) $ with $x^{\prime }\in \left\vert \mathcal{C}%
_{2}^{\prime }\right\vert \cap T^{\prime -n}\left\vert \mathcal{C}%
_{2}^{\prime }\right\vert $, $\sigma ^{\prime }\left( x^{\prime },n\right)
\in K.$ Choose $\xi _{2}\in \left( 0,\varepsilon _{2}\right) $ so that if $%
a,b\in K,$ and $\delta \left( a,a^{\prime }\right) <\xi _{2}$, then $\delta
\left( ba,ba^{\prime }\right) <\varepsilon _{2}.$ (This is possible by
invoking the uniform continuity of the group multiplication on $WK,$ where $%
W $ is a compact neighborhood of $e_{G}.)$

Fix a tower $\mathcal{T}_{2,j}$ in $\mathcal{C}_{2}$ and suppose that $%
\left\vert \mathcal{T}_{2,j}\right\vert \cap \left\vert \mathcal{C}%
_{1}\right\vert \neq \emptyset .$ Let $\phi _{2}:A_{2,j,1}\rightarrow
A_{2,j,1}^{\prime }$ be an isomorphism. For each level $A_{2,j,m}\subset
\left\vert \mathcal{T}_{2,j}\right\vert $ define a function $%
q_{m}:A_{2,j,1}\rightarrow \mathbb{N}$ so that $T^{q_{m}}:A_{2,j,1}%
\rightarrow A_{2,j,m}$ isomorphically, and for almost all $x\in A_{2,j,1},$ 
\begin{equation*}
\sigma ^{\prime }\left( \phi _{2,j}\left( x\right) ,m\right) \sigma \left(
x,q_{m}\left( x\right) \right) ^{-1}\in B\left( \xi _{2},e_{G}\right) .
\end{equation*}%
The $q_{m}$ are chosen so that $q_{m+1}>q_{m}$ and so that, as before,
defining $p_{2}$ on each $A_{2,j,m}$ by 
\begin{equation*}
p_{2}\left( x\right) =q_{m+1}\left( T^{-q_{m}}\left( x\right) \right)
-q_{m}\left( T^{-q_{m}}\left( x\right) \right)
\end{equation*}%
the transformation $T_{2}\left( x\right) =T^{p_{2}}\left( x\right) $ is a
speedup of $T$ and agrees with $T_{1}$ on its domain. This can be done by
repeated application of Lemma \ref{push forward function}$.$ Explicitly, if $%
A_{2,j,1}\not\subset \left\vert \mathcal{C}_{1}\right\vert ,$ then by Lemma %
\ref{push forward function} there is a function $q_{1}:A_{2,j,1}\rightarrow 
\mathbb{N}$ so that $T^{q_{1}}:A_{2,j,1}\rightarrow A_{2,j,2}$
isomorphically, and for almost all $x\in A_{2,j,1}$%
\begin{equation*}
\sigma ^{\prime }\left( \phi _{2,j}\left( x\right) ,1\right) \sigma \left(
x,q_{1}\left( x\right) \right) ^{-1}\in B\left( \xi _{2},e_{G}\right) .
\end{equation*}%
If $A_{2,j,2}$ $\not\subset \left\vert \mathcal{C}_{1}\right\vert ,$ then we
choose $q_{2}>q_{1}$ on $A_{2,j,1}$ so that $T^{q_{2}}:A_{2,j,1}\rightarrow
A_{2,j,3}$ isomorphically, and for almost all $x\in A_{2,j,1}$%
\begin{equation*}
\sigma ^{\prime }\left( \phi _{2,j}\left( x\right) ,2\right) \sigma \left(
x,q_{2}\left( x\right) \right) ^{-1}\in B\left( \xi _{2},e_{G}\right) .
\end{equation*}%
We continue in this way until (unless) we first arrive at a level $%
A_{2,j,m}\subset \left\vert \mathcal{C}_{1}\right\vert .$ There $T_{1}$ is
already defined, and we let $q_{m+1}=q_{m}+p_{1}.$ We continue this way
until we reach the top level $A_{2,j,m+h_{1}-1}$ of this $\mathcal{C}_{1}-$%
column. If there is another level $A_{2,j,m+h_{1}}$ of $\left\vert \mathcal{T%
}_{2,j}\right\vert ,$ we define $q_{m+h_{1}}$ as before and continue until
all levels of $\left\vert \mathcal{T}_{2,j}\right\vert $ have been addressed.

Having defined $T_{2}$ on $\left\vert \mathcal{T}_{2,j}\right\vert ^{o},$ $%
\phi _{2,j}$ is then extendible uniquely to $\left\vert \mathcal{T}%
_{2,j}\right\vert $ by the requirement that for almost all $x\in \left\vert 
\mathcal{T}_{2,j}\right\vert ^{o},$ 
\begin{equation*}
\phi _{2,j}\left( T_{2}x\right) =T^{\prime }\left( \phi _{2,j}x\right) .
\end{equation*}

Let $\sigma _{2}$ denote the cocyle determined by $T_{2}$ and $\sigma ,$
which is defined for pairs $\left( x,n\right) ,$ where $x\in \left\vert 
\mathcal{T}_{2,j}\right\vert \cap T_{2}^{-n}\left\vert \mathcal{T}%
_{2,j}\right\vert .$ We define $\alpha _{2}:\left\vert \mathcal{T}%
_{2,j}\right\vert \rightarrow G$ in two stages. First, for $x\in
A_{2,j,m}\subset \left\vert \mathcal{T}_{2,j}\right\vert \cap \left\vert 
\mathcal{C}_{1}\right\vert $, if $x\in A_{1,j,l+1}$ (that is, $x$ is in the $%
\left( l+1\right) ^{st}$ level of $\mathcal{C}_{1}$), then we let 
\begin{equation*}
\bar{\alpha}_{1}\left( x\right) =\sigma ^{\prime }\left( \phi _{2,j}\left(
x\right) ,-l\right) ^{-1}\sigma _{2}^{a_{1}}\left( x,-l\right) .
\end{equation*}%
We set $\bar{\alpha}_{1}\left( x\right) =e_{G}$ on $\left\vert \mathcal{T}%
_{2,j}\right\vert \backslash \left\vert \mathcal{C}_{1}\right\vert $. We see
that for $x\in \left\vert \mathcal{T}_{2,j}\right\vert \cap \left\vert 
\mathcal{C}_{1}\right\vert ,$%
\begin{eqnarray*}
\delta \left( \bar{\alpha}_{1}\left( x\right) ,e_{G}\right) &=&\delta \left(
\sigma ^{\prime }\left( \phi _{2,j}\left( x\right) ,-l\right) ^{-1},\sigma
_{2}^{a_{1}}\left( x,-l\right) ^{-1}\right) \\
&\leq &\delta \left( \sigma ^{\prime }\left( \phi _{2,j}\left( x\right)
,-l\right) ^{-1},\sigma ^{\prime }\left( \phi _{1}\left( x\right) ,-l\right)
\right) \\
&&+\delta \left( \sigma ^{\prime }\left( \phi _{1}\left( x\right) ,-l\right)
,\sigma _{2}^{a_{1}}\left( x,-l\right) ^{-1}\right) \\
&=&\delta \left( \sigma ^{\prime }\left( \phi _{2,j}\left( x\right)
,-l\right) ^{-1},\sigma ^{\prime }\left( \phi _{1}\left( x\right) ,-l\right)
\right) \\
&\leq &2\varepsilon _{2}
\end{eqnarray*}%
where the last inequality follows from the condition on the near constancy
of $\sigma ,$ on the columns of $\mathcal{C}_{2}^{\prime }.$ Thus $\delta
\left( \bar{\alpha}_{1}\left( x\right) \alpha \left( x\right) ,\alpha \left(
x\right) \right) =\delta \left( \bar{\alpha}_{1}\left( x\right)
,e_{G}\right) \leq 2\varepsilon _{2}.$ Moreover, on a $\mathcal{C}_{1}-$%
column contained in $\mathcal{T}_{2,j},$ the map $\left( x,g\right) \mapsto
\left( \phi _{2,j}\left( x\right) ,\bar{\alpha}_{1}\left( x\right) \alpha
_{1}\left( x\right) g\right) $ is a $G-$isomorphism from $\left(
T_{2}\right) _{\sigma _{2}}$ to $T_{\sigma ^{\prime }}^{\prime }.$

Now for any $x\in A_{2,j,m+1}\subset \left\vert \mathcal{T}_{2,j}\right\vert 
$ let 
\begin{equation*}
\tilde{\alpha}_{1}\left( x\right) =\sigma ^{\prime }\left( \phi _{2,j}\left(
x\right) ,-m\right) ^{-1}\sigma _{2}^{\bar{\alpha}_{1}\alpha _{1}}\left(
x,-m\right) .
\end{equation*}%
We see that $\delta \left( \tilde{\alpha}_{1}\left( x\right) ,e_{G}\right)
<\varepsilon _{2}.$ Indeed, if $x\in \left\vert \mathcal{T}_{2,j}\right\vert
\backslash \left\vert \mathcal{C}_{1}\right\vert ,$ or if $x$ is in $%
\left\vert \mathcal{T}_{2,j}\right\vert $ and in the base of $\mathcal{C}%
_{1},$ then this is immediate from the construction$.$ On the other hand,
suppose $x\in \left\vert \mathcal{C}_{1}\right\vert $ but not in the base of 
$\mathcal{C}_{1}.$ Say $x$ is in $A_{2,j,m+1}$ and in the $\left( l+1\right)
^{st}$ level of $\mathcal{C}_{1}.$ Then we have 
\begin{equation*}
\delta \left( \tilde{\alpha}_{1}\left( x\right) ,e_{G}\right) =\delta \left(
\sigma ^{\prime }\left( \phi _{2,j}\left( x\right) ,-m\right) ^{-1},\sigma
_{2}^{\bar{\alpha}_{1}\alpha _{1}}\left( x,-m\right) ^{-1}\right) .
\end{equation*}%
But%
\begin{equation*}
\sigma ^{\prime }\left( \phi _{2,j}\left( x\right) ,-m\right) ^{-1}=\sigma
^{\prime }\left( \phi _{2,j}\left( x\right) ,-l\right) ^{-1}\sigma ^{\prime
}\left( T^{\prime -l}\left( \phi _{2,j}\left( x\right) \right) ,-m\right)
^{-1}
\end{equation*}%
and 
\begin{equation*}
\sigma _{2}^{\bar{\alpha}_{1}\alpha _{1}}\left( x,-m\right) ^{-1}=\sigma
_{2}^{\bar{\alpha}_{1}\alpha _{1}}\left( x,-l\right) ^{-1}\sigma _{2}^{\bar{%
\alpha}_{1}\alpha _{1}}\left( T_{2}^{-l}\left( x\right) ,-m\right) ^{-1}.
\end{equation*}%
Furthermore, 
\begin{equation*}
\sigma ^{\prime }\left( \phi _{2,j}\left( x\right) ,-l\right) =\sigma _{2}^{%
\bar{\alpha}_{1}\alpha _{1}}\left( x,-l\right)
\end{equation*}%
and 
\begin{equation*}
\delta \left( \sigma ^{\prime }\left( T^{\prime -l}\left( \phi _{2,j}\left(
x\right) \right) ,-m\right) ^{-1},\sigma _{2}^{\bar{\alpha}_{1}\alpha
_{1}}\left( T_{2}^{-l}\left( x\right) ,-m\right) ^{-1}\right) <\xi _{2}
\end{equation*}%
so by the choice of $\xi _{2}$ we conclude that 
\begin{equation*}
\delta \left( \tilde{\alpha}_{1}\left( x\right) ,e_{G}\right) \leq
\varepsilon _{2}.
\end{equation*}

Now set $\alpha _{2}\left( x\right) =\tilde{\alpha}_{1}\left( x\right) \bar{%
\alpha}_{1}\left( x\right) \alpha _{1}\left( x\right) $ and observe that 
\begin{equation*}
\delta \left( \alpha _{2}\left( x\right) ,\alpha _{1}\left( x\right) \right)
\leq 3\varepsilon _{2}
\end{equation*}%
(using the right invariance of $\delta $). The map $\left( x,g\right)
\mapsto \left( \phi _{2,j}\left( x\right) ,\alpha _{2}\left( x\right)
g\right) $ is a $G-$ isomorphism from $\left( T_{2}\right) _{\sigma _{2}}$
to $T_{\sigma ^{\prime }}^{\prime }$ on all of $\left\vert \mathcal{T}%
_{2,j}\right\vert .$

Perform this construction on each tower of $\mathcal{C}_{2}$ which meets $%
\left\vert \mathcal{C}_{1}\right\vert .$ If $\mathcal{T}_{2,j}$ is a tower
of $\mathcal{C}_{2}$ that does not meet $\left\vert \mathcal{C}%
_{1}\right\vert ,$ then employ the simpler construction that was used in the
first stage of the proof to define $T_{2}$ and $\alpha _{2}$ on such a
tower. Setting $\alpha _{2}\left( x\right) =e_{G}$ on $X\backslash
\left\vert \mathcal{C}_{2}\right\vert $ completes the second stage of the
proof.

This procedure can be repeated indefinitely to produce a sequence of castles 
$\mathcal{C}_{k}$ in $X$ for partially defined transformations $T_{k}$,
where the levels of $\mathcal{C}_{k}$ approximate the partition $\mathcal{P}%
_{k-1}$ to within $\frac{1}{k},$ so that each $T_{k}$ is a speedup of $T$
defined on $\left\vert \mathcal{C}_{k}\right\vert ^{o},$ each $T_{k+1}$
extends $T_{k},$ and the transformation $\bar{T}=\bigcup_{k}T_{k}$ is a
speedup of $T$ defined almost everywhere. Let $\bar{\sigma}$ denote the
cocycle for $\bar{T}$ that arises from $\sigma $.

The construction also produces a sequence of isomorphisms $\phi
_{k}:\left\vert \mathcal{C}_{k}\right\vert \rightarrow \left\vert \mathcal{C}%
_{k}^{\prime }\right\vert $ that intertwine $T_{k}$ and $T^{\prime }.$ In
addition, it produces a sequence of functions $\alpha _{k}:X\rightarrow G$
so that for each $k,$ $\delta \left( \alpha _{k+1},\alpha _{k}\right) \leq
3\varepsilon _{k+1}$ and so that the map $\left( x,g\right) \mapsto \left(
\phi _{k}x,\alpha _{k}g\right) $ is a $G-$isomorphism between $\left(
T_{k}\right) _{\sigma _{k}}$ and $T_{\sigma ^{\prime }}^{\prime }$ on $%
\left\vert \mathcal{C}_{k}\right\vert .$ Since $\delta $ is complete we see
that the sequence $\alpha _{k}$ converges uniformly to a function $\alpha ,$
such that for almost all $x,$ $\delta \left( \alpha \left( x\right)
,e_{g}\right) \leq \varepsilon ,$ and hence $\alpha \left( x\right) \in U.$

In the contruction of the $\phi _{k}$ we observe that each $\phi _{k+1}$
agrees set-wise with $\phi _{k}$ on the levels of $\mathcal{C}_{k}.$ Since
the $\sigma -$algebras $\mathcal{B}_{k}$ generated by the levels of $%
\mathcal{C}_{k}$ increase to the full $\sigma -$algebra, the maps $\phi _{k}$
determine an isomorphism $\phi $ between $\bar{T}$ and $T^{\prime }$ which,
for each $k,$ agrees set-wise with $\phi _{k}$ on $\mathcal{B}_{k}.$

Now we confirm that the map $\left( x,g\right) \mapsto \left( \phi x,\alpha
\left( x\right) g\right) $ is a $G-$isomorphism from $\bar{T}_{\bar{\sigma}}$
to $T_{\sigma ^{\prime }}^{\prime }.$ We want to establish that for each $n,$
\begin{eqnarray*}
\sigma ^{\prime }\left( \phi x,n\right) &=&\alpha \left( \bar{T}^{n}x\right) 
\bar{\sigma}\left( x,n\right) \alpha \left( x\right) ^{-1}\text{ a.e.} \\
&=&\bar{\sigma}^{\alpha }\left( x,n\right) .
\end{eqnarray*}%
Fix $n\in \mathbb{Z}$ and $\eta >0.$ For almost every $x,$ if $k$ is
sufficiently large, then 
\begin{equation*}
\bar{T}^{n}x=T_{k}^{n}x,
\end{equation*}%
\begin{equation*}
\bar{\sigma}\left( x,n\right) =\sigma _{k}\left( x,n\right)
\end{equation*}%
and 
\begin{equation*}
\delta \left( \alpha \left( x\right) ,\alpha _{k}\left( x\right) \right)
\leq \eta .
\end{equation*}%
Furthermore, since the points $\phi \left( x\right) $ and $\phi _{k}\left(
x\right) $ are in the same level of $\mathcal{C}_{k}^{\prime },$ the
approximate constancy of $\sigma ^{\prime }\left( \cdot ,n\right) $ on such
a level gives 
\begin{equation*}
\delta \left( \sigma ^{\prime }\left( \phi x,n\right) ,\sigma ^{\prime
}\left( \phi _{k}x,n\right) \right) \leq \eta .
\end{equation*}%
But we know that 
\begin{eqnarray*}
\sigma ^{\prime }\left( \phi _{k}x,n\right) &=&\alpha _{k}\left(
T_{k}^{n}x\right) \sigma _{k}\left( x,n\right) \alpha _{k}\left( x\right)
^{-1}\text{ a.e.} \\
&=&\alpha _{k}\left( \bar{T}^{n}x\right) \bar{\sigma}\left( x,n\right)
\alpha _{k}\left( x\right) ^{-1}.
\end{eqnarray*}%
We only need to conclude that this value is close to $\bar{\sigma}^{\alpha
}\left( x,n\right) .$

Now we describe the additional condition according to which the subsequence $%
\varepsilon _{n_{2}},$ $\varepsilon _{n_{3}},...$ (relabeled $\varepsilon
_{2},\varepsilon _{3},...$) must be chosen. For each $k\geq 1,$ there is a
fixed compact set $K$ containing all the values of $\alpha _{k}$ and all the
values of $\sigma _{k}\left( x,t\right) ,$ for $x\in \left\vert \mathcal{C}%
_{k}\right\vert \cap T_{k}^{-t}\left\vert \mathcal{C}_{k}\right\vert .$
(Recall that all the values of $\alpha _{k}$ lie in $\bar{B}\left(
\varepsilon ,e_{G}\right) ,$ which is compact.) Regardless of how the $%
\varepsilon _{k}$ will be chosen, we will have, for all $k$ and $x,$ 
\begin{equation*}
\delta \left( \alpha _{k}\left( x\right) ,\alpha \left( x\right) \right)
<\sum_{m=k+1}^{\infty }\varepsilon _{m}.
\end{equation*}%
The additional condition we impose on the $\varepsilon _{k}$ is that this
sum is so small as to ensure that for all $x\in \left\vert \mathcal{C}%
_{k}\right\vert \cap T_{k}^{-t}\left\vert \mathcal{C}_{k}\right\vert $, 
\begin{equation*}
\mathcal{\delta }\left( \alpha _{k}\left( T_{k}^{t}x\right) \sigma
_{k}\left( x,t\right) \alpha _{k}\left( x\right) ^{-1},\alpha \left( \bar{T}%
^{t}x\right) \sigma _{k}\left( x,t\right) \alpha \left( x\right)
^{-1}\right) <\frac{1}{k}.
\end{equation*}%
If this is done, then for the given $n$ and $\eta ,$ arguing with
sufficently large $k$, we can conclude that 
\begin{equation*}
\delta \left( \sigma ^{\prime }\left( \phi x,n\right) ,\alpha \left( \bar{T}%
^{n}x\right) \bar{\sigma}\left( x,n\right) \alpha \left( x\right)
^{-1}\right) <\eta +\frac{1}{k}.
\end{equation*}%
Since $\eta $ is arbitrary, we have the desired equality.\bigskip
\end{proof}

We note that in the case that $G$ is a discrete group, we immediately obtain
the following stronger result:

\begin{corollary}
\label{discrete gp ext}Let $T_{\sigma }$ and $T_{\sigma ^{\prime }}^{\prime
} $ be $G-$extensions of $T$ and $T^{\prime },$ where $G$ is a finite or
countable group. Suppose that $T_{\sigma }$ is ergodic and $T^{\prime }$ is
aperiodic. Then there is a relative speedup of $T_{\sigma }$ which is $G-$%
isomorphic to $T_{\sigma ^{\prime }}^{\prime }$ by a relative isomorphism
whose transfer function $\alpha $ satisfies $\alpha \left( x\right) =e_{G}$
almost everywhere.\bigskip
\end{corollary}

\section{Finite and countable extensions}

We now turn to the analysis of speedups of $n-$point extensions. First we
introduce a simplification of some of our notation. Given an automorphism $T$
and a cocycle $\sigma $ taking values in a group $G,$ we will denote the $G-$%
extension $T_{\sigma }$ more simply by the single letter $S$, and in
general, $G-$extensions $T_{\sigma ^{\prime }}^{\prime }$ or $\left(
T_{1}\right) _{\sigma _{1}}$ will be denoted $S^{\prime }$ and $S_{1},$ and
so on$.$

Now fix an integer $n>1.$ Form the measure space $\left\{ \left[ n\right] ,%
\mathcal{P}\left( \left[ n\right] \right) ,p\right\} $ where $\left[ n\right]
=\left\{ 1,...,n\right\} $, and $p\left( \left\{ i\right\} \right) =\frac{1}{%
n},$ for each $i.$ Making use of the natural action of the symmetric group $%
\mathcal{S}_{n}$ on $\left[ n\right] ,$ each cocycle $\sigma $ for $T$
taking values in $\mathcal{S}_{n}$ determines an automorphism $U$ of $%
\left\{ X\times \left[ n\right] ,\mathcal{B\times C},\mu \times p\right\} $
which has $T$ as a factor. Namely, we have the automorphism $U$ given by 
\begin{equation*}
U^{n}\left( x,i\right) =\left( T^{n}x,\sigma \left( x,n\right) \left(
i\right) \right) .
\end{equation*}

We refer to $U$ as an $n-$point extension of $T.$ (Since we will only
consider ergodic $n-$point extensions, we may restrict ourselves to the
uniform measure $p$). We will use the same sort of notational convention as
above: the $n-$point extensions associated with pairs $\left( T^{\prime
},\sigma ^{\prime }\right) $ and $\left( T_{1},\sigma _{1}\right) $ will be
written $U^{\prime }$ and $U_{1},$ and so on.

Given a pair of $n-$point extensions $\,U_{1}$ and $U_{2}$ on spaces $%
X_{1}\times \left[ n\right] $ and $X_{2}\times \left[ n\right] ,$ we say $%
U_{1}$ is relatively isomorphic to $U_{2}$ if there is an isomorphism $\Phi $
from $U_{1}$ to $U_{2}$ that preserves the fibers of these extensions. That
is, there is an isomorphism of the form 
\begin{equation*}
\left( x,i\right) \mapsto \left( \phi \left( x\right) ,\alpha \left(
x\right) \left( i\right) \right)
\end{equation*}%
where $\alpha :X_{1}\rightarrow \mathcal{S}_{n}.$ Equivalently, these
extensions are relatively isomorphic if there is an isomorphism $\phi $ from 
$T_{1}$ to $T_{2}$ and a function $\alpha :X_{1}\rightarrow \mathcal{S}_{n}$
such that%
\begin{equation*}
\sigma _{2}\left( \phi x,n\right) =\alpha \left( T_{1}^{n}\left( x\right)
\right) \sigma _{1}\left( x,n\right) \alpha \left( x\right) ^{-1}
\end{equation*}%
which is exactly the condition that the $\mathcal{S}_{n}-$extensions $S_{1}$
and $S_{2}$ are $\mathcal{S}_{n}-$isomorphic. We also note that every
speedup of $T_{1}$ (or equivalently, every $\mathcal{S}_{n}-$speedup of $%
S_{1}$) corresponds to a speedup of $U_{1}$ relative to $T_{1}.$

Given $n-$point extensions $U_{1}$ and $U_{2}$, let us write $%
U_{1}\rightsquigarrow U_{2}$ when there is a speedup of $U_{1}$ relative to $%
T_{1}$ which is relatively isomorphic to $U_{2}.$ This relation is evidently
transitive and apparently asymmetric; there is no reason to suppose that $%
U_{1}\rightsquigarrow U_{2}$ implies $U_{2}\rightsquigarrow U_{1}$. In the
case of ergodic finite group extensions, however, (as well as for more
general locally compact second countable group extensions), we have just
seen that it is symmetric, and in fact for each locally compact second
countable group $G$ there is only one equivalence class of ergodic $G-$%
extensions. But, for general ergodic $n-$point extensions, we will see that
the relation is indeed asymmetric. This is due to the fact that the
associated $\mathcal{S}_{n}-$ extensions, of which the ergodic $n-$point
extensions are factors, need not themselves be ergodic. By examining the
ergodic components of these $\mathcal{S}_{n}-$ extensions, we will obtain a
characterization of this relation in other terms, and we will give an
explicit example to illustrate its asymmetry.

Fix an automorphism $T$ and an $\mathcal{S}_{n}-$cocycle $\sigma .$ Let $S$
be the associated $\mathcal{S}_{n}-$ extension of $T.$ We associate to the
pair $\left( T,\sigma \right) $ a conjugacy class of subgoups of $\mathcal{S}%
_{n}$ that will be the basis of the characterization. We recall the
discussion that can be found in $\left[ G\right] $: Let $C$ be an ergodic
component of $S.$ For each $x\in X,$ let $C_{x}=\left\{ \gamma \in \mathcal{S%
}_{n}:\left( x,\gamma \right) \in C\right\} .$ (By the ergodicity of $T,$ $%
\left\vert C_{x}\right\vert \geq 1$ is a constant). Then if $\beta
:X\rightarrow \mathcal{S}_{n}$ is any measurable function such that $\left(
x,\beta \left( x\right) \right) \in C$ almost everywhere, there\ is a
subgroup $G$ of $\mathcal{S}_{n}$ such that the sets $\beta \left( x\right)
^{-1}C_{x}$ are almost all equal to $G.$ Moreover, letting $\alpha \left(
x\right) =\beta \left( x\right) ^{-1},$ and defining a new cocycle by $%
\sigma ^{\prime }\left( x,n\right) =\alpha \left( T^{n}x\right) \sigma
\left( x,n\right) \alpha \left( x\right) ^{-1},$\ we get a new $\mathcal{S}%
_{n}-$extension $S^{\prime }$ that is $\mathcal{S}_{n}-$isomorphic to $S.$
The map 
\begin{equation*}
\left( x,\gamma \right) \mapsto \left( x,\alpha \left( x\right) \gamma
\right)
\end{equation*}%
is an $\mathcal{S}_{n}-$isomorphism from $S$ to $S^{\prime },$ which carries 
$C$ to $X\times G,$ on which $S^{\prime }$ is ergodic. In summary: $\sigma $
is cohomologous to a $G-$valued cocycle $\sigma ^{\prime }$ yielding an $%
\mathcal{S}_{n}-$extension $S^{\prime }$ that has $X\times G$ as an ergodic
component. (In particular, the values of $\sigma ^{\prime }$ lie in $G$).
The groups $G$ that fit this description form a conjugacy class of subgroups
of $\mathcal{S}_{n}.$ We denote this conjugacy class by $gp\left( T,\sigma
\right) .$

We can easily extend the above discussion to a slightly more general
context: Suppose that, for some subgroup $G\subset \mathcal{S}_{n},$ $%
X\times G$ is $S-$invariant, but is not an ergodic component of $S$. Then we
can apply the above arguments to an ergodic component of $S$ that is
contained in $X\times G,$ and conclude that there is a subgroup $H\subset G$
and a cocycle $\sigma ^{\prime }$ cohomologous to $\sigma $ via a $G-$valued
transfer function, so that $X\times H$ is an ergodic component of the $%
\mathcal{S}_{n}-$extension $S^{\prime }$ associated with $\sigma ^{\prime }$.

For brevity, when $X\times G$ is an ergodic component of an $\mathcal{S}%
_{n}- $extension $S,$ we will say that $\sigma $ is $G-$ergodic for $T.$
(This is equivalent to the \textquotedblleft $G-$interchange
property\textquotedblright\ that Gerber introduced in \cite{G}).

It is clear that $gp\left( T,\sigma \right) $ is an invariant of factor
isomorphism. In \cite{G} Gerber showed that it is a complete invariant for
factor orbit equivalence of ergodic $n-$point extensions. In connection with
speedups, we now prove:

\begin{theorem}
\label{nptext}Let $U_{1}$ and $U_{2}$ be ergodic $n-$point extensions of
transformations $\left( T_{1},X_{1}\right) $ and $\left( T_{2},X_{2}\right) $
by $\mathcal{S}_{n}-$valued cocycles $\sigma _{1}$ and $\sigma _{2}.$ Then $%
U_{1}\rightsquigarrow U_{2}$ if and only if for some $G_{1}\in gp\left(
T_{1},\sigma _{1}\right) \ $(and hence for every $G_{1}\in gp\left(
T_{1},\sigma _{1}\right) $), there exists $G_{2}\in gp\left( T_{2},\sigma
_{2}\right) $ such that $G_{2}\subset G_{1}.$
\end{theorem}

\begin{proof}
Suppose first that for some $G_{1}\in gp\left( T_{1},\sigma _{1}\right) ,\ $%
there exists $G_{2}\in gp\left( T_{2},\sigma _{2}\right) $ such that $%
G_{2}\subset G_{1}.$ By the above discussion, for each $i=1,2,$ there is a
cocycle $\sigma _{i}^{\prime }$ cohomologous to $\sigma _{i},$ so that $%
X_{i}\times G_{i}$ is an ergodic component of $\left( T_{i}\right) _{\sigma
_{i}}$. Therefore, without loss of generality, we may assume from the start
that each $\left( T_{i},\sigma _{i}\right) $ has this property.

$S_{1}$ induces an ergodic transformation $\left( S_{1}\right) _{X_{1}\times
G_{2}}$ on $\left( X_{1}\times G_{2}\right) $ and, for each $\left(
x,g_{2}\right) \in X_{1}\times G_{2}$ we let $j=j\left( x,g_{2}\right) $
denote the first return time of $\left( x,g_{2}\right) $ to $X_{1}\times
G_{1}$ under $S_{1}.$ But $j\left( x,g_{2}\right) $ depends only on $x$
since, for all $\left( x,g_{2}\right) \in X_{1}\times G_{2}$ and all $n\in 
\mathbb{N}$, 
\begin{equation*}
S_{1}^{n}\left( x,g_{2}\right) \in \left( X_{1}\times G_{2}\right) \iff
\sigma _{1}\left( x,n\right) g_{2}\in G_{2}\iff \sigma _{1}\left( x,n\right)
\in G_{2}
\end{equation*}%
so that these conditions do not depend on $g_{2}.$ We can then write $%
j\left( x,g_{2}\right) =j\left( x\right) ,\ $and the induced automorphism 
\begin{equation*}
\left( S_{1}\right) _{X_{1}\times G_{2}}=S_{1}^{j}\upharpoonright
_{X_{1}\times G_{2}}
\end{equation*}%
is an ergodic $G_{2}-$extension of the speedup $T_{1}^{j}$ of $T_{1}.$

Applying Theorem \ref{gp ext unif}$,$ we know that there is a $G_{2}-$%
speedup $\left( S_{1}^{j}\upharpoonright _{X_{1}\times G_{2}}\right) ^{k}$
of $S_{1}^{j}\upharpoonright _{X_{1}\times G_{2}}$ , that is $G_{2}-$%
isomorphic to $S_{2}\upharpoonright _{X_{2}\times G_{2}}.$ This gives us a
speedup $\left( T_{1}^{j}\right) ^{k}=T_{1}^{l}$ of $T_{1}$, and $\left(
S_{1}^{j}\upharpoonright _{X_{1}\times G_{2}}\right)
^{k}=S_{1}^{l}\upharpoonright _{X_{1}\times G_{2}}$is a $G_{2}-$extension of 
$T_{1}^{l}$.

There is an isomorphism $\phi :X_{1}\rightarrow X_{2}$ and $\alpha
:X_{1}\rightarrow G_{2}$ so that 
\begin{equation}
T_{1}^{l}\overset{\phi }{\approx }T_{2}\text{ and }\sigma _{2}\phi =\left(
\sigma _{1}^{\left( l\right) }\right) ^{\alpha }.  \label{isom}
\end{equation}

This construction gives us a speedup $U_{1}^{l}$of $U_{1}$ relative to $%
T_{1} $ and conditions $\left( \text{\ref{isom}}\right) $ say that $%
U_{1}^{l} $ is relatively isomorphic to $U_{2}.$

Now suppose that $U_{1}\rightsquigarrow U_{2}.$ Fix $G_{1}\in gp\left(
S_{1}\right) .$ We want to show that there is a subgroup $G_{2}\in sp\left(
S_{2}\right) $ contained in $G_{1}.$ As before, we may assume that $%
X_{1}\times G_{1}$ is an ergodic component of $S_{1}.$ The condition $%
U_{1}\rightsquigarrow U_{2}$ tells us that there is an $\mathcal{S}_{n}-$%
speedup $S_{1}^{k}$ of $S_{1}$ that is $\mathcal{S}_{n}-$isomorphic to $%
S_{2}.$ Since $X_{1}\times G_{1}$ is invariant for $S_{1},$ it remains so
for $S_{1}^{k}.$ Arguing as before, there is a subgroup $G_{2}\subset G_{1}$
and a cocycle $\bar{\sigma}$ for $T_{1}^{k},$ cohomologous to $\sigma
_{1}^{\left( k\right) },$ such that $X_{1}\times G_{2}$ is an ergodic
component of $\bar{S},$ where $\bar{S}$ is the $\mathcal{S}_{n}-$extension
of $T_{1}^{k}$ given by $\bar{\sigma}.$ But, since $\bar{S}$ is relatively
isomorphic to $S_{2},$ we must have $G_{2}\in gp\left( T_{2},\sigma
_{2}\right) .$\bigskip
\end{proof}

As an example, we consider a pair of $3-$point extensions considered by
Gerber in \cite{G}. Let $\left\{ \gamma _{i}\right\} _{i=1}^{6}$ be an
enumeration of the symmetric group $\mathcal{S}_{3}$ so that $\left\{ \gamma
_{i}\right\} _{i=1}^{3}$ is the alternating group $\mathcal{A}_{3}$. Let $%
\left( T_{1},X_{1}\right) $ be the full $3-$shift, with independent
generator $\mathcal{P}=\left\{ P_{i}\right\} _{i=1}^{3}$ and $\left(
T_{2},X_{2}\right) $ the full $6-$shift with independent generator $\mathcal{%
Q}=\left\{ Q_{i}\right\} _{i=1}^{6}.$ Define a cocycle $\sigma _{1}$ for $%
T_{1}$ by setting $\sigma \left( x,1\right) =\gamma _{i}$ when $x\in P_{i}.$
Define $\sigma _{2}$ for $T_{2}$ by setting $\sigma _{2}\left( x,1\right)
=\gamma _{i}$ when $x\in Q_{i}$. In \cite{G} Gerber showed that $gp\left(
S_{1}\right) =\left\{ \mathcal{A}_{3}\right\} $ and $gp\left\{ S_{2}\right\}
=\left\{ \mathcal{S}_{3}\right\} $ and that consequently $S_{1}$ and $S_{2}$
are not factor orbit equivalent. Using Theorem \ref{nptext} we can conclude
further that there is a factor speedup of $S_{2}$ that is factor isomorphic
to $S_{1},$ but there is no factor speedup of $S_{1}$ that is factor
isomorphic to $S_{2}.$\bigskip 

As a final application we observe how the classification of extensions
changes when we pass to extensions that are as close as possible to the
finite case. That is, we consider extensions with countable fibers, and
allow only the smallest natural group of permutations on the fibers. Let $p$
denote counting measure on $\mathbb{N}$ and $\mathcal{S}_{f}$ the group of
finitely supported permutations of $\mathbb{N}$. Suppose that $T$ is an
automorphism of the Lebesgue probability space $\left( X,B,\mu \right) ,$
and $\sigma $ is an $\mathcal{S}_{f}-$cocycle for $T.$ Then as in the finite
case, we obtain a countable extension $U$ of $T$ on $X\times \mathbb{N}$
given by%
\begin{equation*}
U^{n}\left( x,k\right) =\left( T^{n}x,\sigma \left( x,n\right) \left(
k\right) \right) .
\end{equation*}%
Here we say that two such countable extensions $U_{1}$ and $U_{2}$ are
relatively isomorphic if there is an isomorphism $\phi $ from $T_{1}$ to $%
T_{2}$ and a function $\alpha :X\rightarrow \mathcal{S}_{f}$ such that%
\begin{equation*}
\sigma _{2}\left( \phi x,n\right) =\alpha \left( T_{1}^{n}\left( x\right)
\right) \sigma _{1}\left( x,n\right) \alpha \left( x\right) ^{-1}.
\end{equation*}%
Since $\mathcal{S}_{f}$ is countable, Theorem \ref{gp ext unif} applies, and
the analysis of finite extensions which was given above can easily be
adapted to this case. Thus, to each ergodic countable extension $U$ of $T$
by an $\mathcal{S}_{f}-$valued cocycle $\sigma $, we associate a conjugacy
class $gp\left( T,\sigma \right) $ of subgroups of $\mathcal{S}_{f},$ and
the corresponding statement of Theorem \ref{nptext} holds.

\begin{proposition}
Given an ergodic $T,$ there is an uncountable family of ergodic countable
extensions of $T$, each of which acts on the fibers of the extension by
finitely supported permuations of $\mathbb{N}$, and no one of which admits a
relative speedup that is relatively isomorphic to another.
\end{proposition}

\begin{proof}
We first construct an uncountable family of subgroups of $\mathcal{S}_{f}$
such that $\left( i\right) ,$ each acts transitively on $\mathbb{N},$ and $%
\left( ii\right) ,$ no conjugate of one contains another. Suppose $\mathcal{%
P=}\left\{ A_{k}\right\} _{k=1}^{\infty }$ is a partition of $\mathbb{N}$
into two-element sets. Let 
\begin{equation*}
G_{\mathcal{P}}=\left\{ \pi \in \mathcal{S}_{f}:\left( \forall k\right)
\left( \exists j\right) \pi \left( A_{k}\right) =A_{j}\right\} .
\end{equation*}%
It is clear that $G_{\mathcal{P}}$ acts transitively on $\mathbb{N}$. For $%
\xi \in \mathcal{S}_{f}$ we write $\xi \mathcal{P=}\left\{ \xi A\mid A\in 
\mathcal{P}\right\} $ and\ observe that $\xi G_{\mathcal{P}}\xi ^{-1}=G_{\xi 
\mathcal{P}}.$ Hence, if $\mathcal{P}^{\prime }$ is another partition of $%
\mathbb{N}$ into two-element sets and the symmetric difference of $\mathcal{P%
}$ and $\mathcal{P}^{\prime }$ is infinite, then no conjugate of $G_{%
\mathcal{P}}$ can contain $G_{\mathcal{P}^{\prime }}.$ It is easy to
construct an uncountable family $\left\{ \mathcal{P}_{i}\right\} _{i\in I}$
of such partitions so that each pair has an infinite symmetric difference.
The corresponding family of subgroups $\left\{ G_{\mathcal{P}_{i}}\right\}
_{i\in I}$ satisfies conditions $\left( i\right) $ and $\left( ii\right) $
above.

Each of the groups $G_{\mathcal{P}_{i}}$ is countable and amenable, and
hence, by the result of Herman \cite{H}$,$ for each $G_{P_{i}}$ there is a
cocycle $\sigma _{i}$ for $T$ for which the corresponding $G_{\mathcal{P}%
_{i}}-$extension $S_{i}$ is ergodic. But since $G_{\mathcal{P}_{i}}$ acts
transitively on $\mathbb{N}$, the corresponding countable extension $U_{i}$
is also ergodic. Indeed, for all sets $A$ and $B$ contained in $X$ of
positive measure, and all $l,m\in \mathbb{N},$ if we choose $\pi \in
G_{P_{i}}$ such that $\pi \left( l\right) =m,$ then Lemma $\ref{little push}$
gives an $n^{\prime }\in \mathbb{N}$ and $A^{\prime }\subset A$ with $\mu
\left( A^{\prime }\right) >0$ such that $T^{n^{\prime }}\left( A^{\prime
}\right) \subset B$ and for all $x\in A^{\prime },$ $\sigma \left(
x,n^{\prime }\right) =\pi .$ Condition $\left( ii\right) $ on the groups $G_{%
\mathcal{P}_{i}}$ tells us that for all $i\neq j$ in $I,$ no speedup of $%
U_{i}$ relative to $T$ is relatively isomorphic to $U_{j}.$
\end{proof}

\end{document}